\documentclass[letterpaper,reqno,11pt]{amsart}
\usepackage[margin=1.2in]{geometry}

\usepackage{amssymb,amsmath,amsthm, amscd, enumerate, mathrsfs}
\usepackage{graphicx, hhline}
\usepackage[all]{xy}
\usepackage[usenames]{color}
\usepackage{hyperref}
\hypersetup{colorlinks=true}

\theoremstyle{plain}
\newtheorem{thm}{Theorem}[section] 
\newtheorem{cor}[thm]{Corollary}
\newtheorem{prop}[thm]{Proposition}
\newtheorem{conj}[thm]{Conjecture}
\newtheorem{lem}[thm]{Lemma}
\newtheorem*{mainthm}{Main Theorem}
\newtheorem*{conjC}{Conjecture $\mathrm{N}_n$}
\newtheorem*{conjC'}{Conjecture $\mathrm{N}_n'$}
\newtheorem*{conjD}{Conjecture $\mathrm{H}_n$}
\theoremstyle{definition} 
\newtheorem{defn}[thm]{Definition}
\newtheorem{eg}[thm]{Example} 
\theoremstyle{remark}
\newtheorem{rem}[thm]{Remark}

\newtheorem*{cl}{Claim}

\newtheorem*{acknowledgement}{Acknowledgments}

\newcommand{\sO}{\mathcal{O}}
 
\newcommand{\Z}{\mathbb{Z}}
\newcommand{\N}{\mathbb{N}} 
\newcommand{\Q}{\mathbb{Q}} 
 
\newcommand{\C}{\mathbb{C}}
\newcommand{\F}{\mathbb{F}}
\newcommand{\pP}{\mathbb{P}}

\newcommand{\m}{\mathfrak{m}}
\newcommand{\n}{\mathfrak{n}}

\newcommand{\Spec}{\mathop{\mathrm{Spec}}\nolimits}
\newcommand{\Proj}{\mathop{\mathrm{Proj}}\nolimits}

\newcommand{\Ker}{\mathop{\mathrm{Ker}}\nolimits}

\newcommand{\bigzerol}{\smash{\lower1.0ex\hbox{\bg 0}}}

\newcommand{\sbt}{\,\begin{picture}(-1,1)(-1,-3)\circle*{3}\end{picture}\ }

\renewcommand{\labelenumi}{\rm{(\theenumi)}}
\newfont{\bg}{cmr17 scaled\magstep5}

\numberwithin{equation}{section}

\title{Nilpotence of Frobenius action and the Hodge filtration on local cohomology}

\dedicatory{Dedicated to Professor~Kei-ichi Watanabe on the occasion of his retirement.}

\author{Vasudevan Srinivas}
\address{School of Mathematics, Tata Institute of Fundamental Research, Homi Bhabha Road, Colaba, Mumbai 400005, India}
\email{srinivas@math.tifr.res.in}

\author{Shunsuke Takagi}
\address{Graduate School of Mathematical Sciences, University of Tokyo, 3-8-1 Komaba, Meguro-ku, Tokyo 153-8914, Japan}
\email{stakagi@ms.u-tokyo.ac.jp}

\keywords{$F$-nilpotent singularities, Hodge filtration, divisor class groups, Brauer groups}
\subjclass[2010]{13A35, 13C20, 14C30, 14F22}

\begin{document}

\maketitle
\markboth{V.~SRINIVAS and S.~TAKAGI}{NILPOTENCE OF FROBENIUS ACTION}


\begin{abstract}
An $F$-nilpotent local ring is a local ring $(R, \m)$ of prime characteristic defined by the nilpotence of the Frobenius action on its local cohomology modules $H^i_{\m}(R)$. 
A singularity in characteristic zero is said to be of $F$-nilpotent type if its modulo $p$ reduction is $F$-nilpotent for almost all $p$. 
In this paper, we give a Hodge-theoretic interpretation of three-dimensional normal isolated singularities of $F$-nilpotent type. 
In the graded case, this yields a characterization of these singularities in terms of divisor class groups and Brauer groups. 
\end{abstract}

\section*{Introduction}
$F$-singularities, classes of singularities in positive characteristic defined in terms of the Frobenius morphism, have been studied intensively in recent years. 
They partly conjecturally correspond to singularities in birational geometry in characteristic zero via reduction from characteristic zero to positive characteristic. 
Three classes of $F$-singularities appear in this paper: $F$-rationality, $F$-injectivity and $F$-nilpotence. 
As the name suggests, it follows from a combination of results of Smith \cite{Sm}, Hara \cite{Ha} and Mehta-Srinivas \cite{MS} that a singularity in characteristic zero is a rational singularity if and only if its modulo $p$ reduction is $F$-rational for almost all $p$. 
$F$-injectivity, which is a broader class of singularities than $F$-rationality, is defined by the injectivity of the Frobenius action on its local cohomology modules.
It is conjectured that a singularity in characteristic zero is a Du Bois singularity, which has its origin in Hodge theory, if and only if its modulo $p$ reduction is $F$-injective for infinitely many $p$ (the ``if" part was proved by Schwede \cite{Sc}). 
This conjecture is open even in dimension two, and it has recently turned out in \cite{BST} (see also \cite{MuS}) that it is equivalent to another more arithmetic and wide open conjecture, the so-called weak ordinarity conjecture (see Conjecture \ref{conjWO} for the precise statement). 
Since we do not know how to prove these conjectures at this point, we shift our focus on another class of $F$-singularities, $F$-nilpotence, which is also a broader class of singularities than $F$-rationality. 

Let $(R, \m)$ be a $d$-dimensional excellent normal local ring of positive characteristic $p>0$. 
Then, by a result of Smith \cite{Sm}, there exists a unique maximal proper submodule $0^*_{H^d_{\m}(R)}$ of  the local cohomology module $H^d_{\m}(R)$ stable under the natural Frobenius action.
We say that $R$ is $F$-nilpotent if the natural Frobenius actions on $H^2_{\m}(R), \dots, H^{d-1}_{\m}(R)$ and on $0^*_{H^d_{\m}(R)}$ are all nilpotent. 
It is easy to see that $R$ is $F$-rational if and only if it is $F$-injective and $F$-nilpotent. 
To the best of authors' knowledge, the notion of $F$-nilpotence first appeared in 
\cite{BB}\footnote{In \cite{BB}, the term ``close to $F$-rational" is used for the $F$-nilpotence of singularities.} in order to describe some invariants of singularities in positive characteristic, the so-called Lyubeznik numbers, in terms of \'{e}tale cohomology. 
In this paper, we pursue a geometric interpretation of $F$-nilpotence in the case of isolated singularities. 
We say that a singularity $(x \in X)$ in characteristic zero is of $F$-nilpotent type if its modulo $p$ reduction $(x_p \in X_p)$ is $F$-nilpotent for almost all $p$. 
We then propose the following Hodge-theoretic conjecture. 

\begin{conjD}
Let $(x \in X)$ be an $n$-dimensional normal isolated singularity over the field $\C$ of complex numbers $($and then $H^*_{\{x\}}(X_{\rm an}, \C)$ has a canonical mixed Hodge structure due to Steenbrink \cite{St}$)$. 
Then $(x \in X)$ is of $F$-nilpotent type if and only if the zeroth graded piece $\mathrm{Gr}^0_F H^{i}_{\{x\}}(X_{\rm an}, \C)$ of the Hodge filtration vanishes for all $i$.  
\end{conjD}

In the setting of Conjecture $\mathrm{H}_n$, let $\pi:Y \to X$ be a resolution of singularities with simple normal crossing exceptional divisor $E=\sum_i E_i$ such that $\pi|_{Y \setminus E}: Y \setminus E \to X \setminus \{x\}$ is an isomorphism.  
Then the condition that $\mathrm{Gr}^0_F H^{i}_{\{x\}}(X_{\rm an}, \C)=0$ for all $i$ is equivalent to saying that $H^i(E, \sO_E)=0$ for all $i \ge 1$. 
On the other hand, the $F$-nilpotence of the modulo $p$ reduction $(x_p \in X_p)$ of $(x \in X)$ is equivalent, after extending the residue field to its algebraic closure, to the injectivity of the map $\mathrm{id}-F:H^i_{\m_{x_p}}(\sO_{X_p, x_p}) \to H^i_{\m_{x_p}}(\sO_{X_p, x_p})$ for all $i$, where $\mathrm{id}$ is the identity map and $F$ is the natural Frobenius action on the local cohomology module $H^i_{\m_{x_p}}(\sO_{X_p, x_p})$.  
Therefore, making use of the Artin--Schreier sequence in the \'{e}tale topology, we can see that $(x \in X)$ is of $F$-nilpotent type if and only if the natural Frobenius action on the sheaf cohomology $H^i(E_p, \sO_{E_p})$ of the modulo $p$ reduction $E_p$ of $E$ is nilpotent for all $i \ge 1$ and almost all $p$.  

Here we introduce another more arithmetic conjecture, which is closely related to the weak ordinarity conjecture. 

\begin{conjC}
Let $V$ be an $n$-dimensional projective simple normal crossing variety over an algebraically closed field of characteristic zero. 
Then there exist infinitely many primes $p$ such that for each $i$, the natural Frobenius action on the cohomology group $H^i(V_{p}, \sO_{V_{p}})$ of the modulo $p$ reduction $V_p$ of $V$ is not nilpotent provided $H^i(V, \sO_V)$ is nonzero. 
\end{conjC}

Although the weak ordinarity conjecture is open even in the case of curves, Conjectures $\mathrm{N}_1$ and $\mathrm{N}_2$ are known to be true, which is essentially due to Serre \cite{Se} and Ogus \cite{Og} (see also \cite{JR}). 
It follows from an application of Conjecture $\mathrm{N}_{n-1}$ to the above $E$ that the natural Frobenius action on $H^i(E_{p}, \sO_{E_{p}})$ is nilpotent for almost all $p$ if and only if $H^i(E, \sO_E)$ vanishes.  
Thus, summing up the above discussion, we obtain the following theorem. 

\begin{mainthm}[\textup{=Theorem \ref{main result}}]
If Conjecture $\mathrm{N}_{n-1}$ holds, then Conjecture $\mathrm{H}_n$ holds as well. 
In particular, Conjecture $\mathrm{H}_3$ holds true. 
\end{mainthm}

As an application, we give a characterization of two-dimensional normal singularities of $F$-nilpotent type in terms of divisor class groups. 
That is, let $(R, \m)$ be a two-dimensional normal local ring essentially of finite type over an algebraically closed field of characteristic zero. 
Then $R$ is of $F$-nilpotent type if and only if the divisor class group $\mathrm{Cl}(\widehat{R})$ of the $\m$-adic completion $\widehat{R}$ does not contain the torsion group $\Q/\Z$ (Theorem \ref{divisor class group}). 
This should be compared with Lipman's characterization of two-dimensional rational singularities in terms of divisor class groups \cite{Li}. 
We also have a similar characterization of three-dimensional normal graded isolated singularities of $F$-nilpotent type in terms of Brauer groups (Theorem \ref{Brauer}). 

\begin{small}
\begin{acknowledgement}
The authors are indebted to Bhargav Bhatt for valuable comments, and from whom they learned Lemma \ref{p-adicHodge}. 
The second author is also grateful to Osamu Fujino and Kazuma Shimomoto for helpful comments on an earlier draft of this paper.  
He would like to thank Taro Fujisawa, Yoshinori Gongyo,  Atsushi Ito, Yujiro Kawamata and Yoichi Mieda for useful discussions. 
The first author was supported by a J. C. Bose Fellowship of the Department of Science and Technology, India.
The second author was partially supported by Grant-in-Aid for Young Scientists (B) 23740024 from JSPS. 
The material is based upon work supported by the National Science Foundation under Grant No. 0932078 000, initiated while the authors were in residence at the Mathematical Science Research Institute (MSRI) in Berkeley, California, during the spring semester 2013.
\end{acknowledgement}
\end{small}


\section{Preliminaries}
In this section, we will recall some definitions and results that we will need in later sections. 
\subsection{Du Bois singularities}
First we recall the definition of Du Bois singularities, originally introduced by Steenbrink to study degenerations of variations of Hodge structures. 

Let $(x \in X)$ be a reduced singularity over an algebraically closed field $k$ of characteristic zero. 
A \textit{log resolution} (resp.~a \textit{projective log resolution}) of $(x \in X)$ is a proper (resp. projective) birational morphism $\pi:Y \to X$ with $Y$ smooth such that the closed fiber $\pi^{-1}(x)$ and the exceptional locus $\mathrm{Exc}(\pi)$ of $\pi$ are divisors on $Y$ and that $\mathrm{Exc}(\pi) \cup \pi^{-1}(x)$ are simple normal crossing divisors. 

\begin{defn}
Let $(x \in X)$ be a singularity over an algebraically closed field of characteristic zero and $\underline{\Omega}_X^0$ be the zeroth graded piece of the Du Bois complex $\underline{\Omega}_X^{\sbt}$ (see \cite{DB} for the original definition of $\underline{\Omega}_X^{0}$ and \cite{Sc2} for its alternative characterization). 
We say that $(x \in X)$ is \textit{Du Bois} if the canonical map $\sO_{X, x} \to \underline{\Omega}_{X,x}^0$ is a quasi-isomorphism. 
\end{defn}

There exists a simple characterization of Du Bois singularities in the case of normal isolated singularities \cite{St}:  
let $(x \in X)$ be a normal isolated singularity over an algebraically closed field of characteristic zero. 
Take a log resolution $\pi:Y \to X$ of $(x \in X)$, and put $E:=\pi^{-1}(x)_{\rm red}$. 
Then $(x \in X)$ is Du Bois if and only if the natural map $(R^if_*\sO_Y)_x \to H^i(E, \sO_E)$ is an isomorphism for all $i \ge 1$. 

\subsection{$F$-rational and $F$-injective rings}
In this subsection, we briefly review the definitions of $F$-rational and $F$-injective rings and their basic properties. 

In this paper, all rings are excellent commutative rings with unity. 
For a ring $R$, we denote by $R^{\circ}$ the set of elements of $R$ which are not in any minimal prime ideal. 
Let $R$ be a ring of prime characteristic $p$ and $F:R \to R$ be the Frobenius map sending $x \in R$ to $x^p \in R$. 
If $(R, \m)$ is local, then the Frobenius map $F$ induces a $p$-linear map $H^i_{\m}(R) \to H^i_{\m}(R)$ for each $i$, which we denote by the same letter $F$. 
The $e$-th iteration of $F$ is denoted by $F^e$. 
\begin{defn}
Let $(R, \m)$ be a $d$-dimensional local ring of characteristic $p>0$. 
\begin{enumerate}[(i)]
\item We say that $R$ is \textit{$F$-injective} if $F: H^i_{\m}(R) \to H^i_{\m}(R)$ is injective for all $i$. 
\item We say that $R$ is \textit{$F$-rational} if $R$ is Cohen-Macaulay and if for any $c \in R^{\circ}$, there exists $e \in \N$ such that $cF^e: H^d_{\m}(R) \to H^d_{\m}(R)$ is injective. 
\end{enumerate}
\end{defn}

It is immediate from definition that $F$-rational rings are $F$-injective. 

\begin{defn}
Let $(R,\m)$ be a $d$-dimensional reduced local ring of characteristic $p>0$. 
\begin{enumerate}[(i)]
\item 
The \textit{Frobenius closure} $0^{F}_{H^i_{\m}(R)}$ of the zero submodule in $H^i_{\m}(R)$ is the submodule of $H^i_{\m}(R)$ consisting of all elements $z \in H^i_{\m}(R)$ for which there exists $e \in \N$ such that $F^e(z)=0$.

\item 
The \textit{tight closure} $0^{*}_{H^d_{\m}(R)}$ of the zero submodule in $H^d_{\m}(R)$ is the submodule of $H^d_{\m}(R)$ consisting of all elements $z \in H^d_{\m}(R)$ for which there exists $c \in R^{\circ}$ such that $cF^e(z)=0$ for all large $e \in \N$. 

\item 
An element $c \in R^{\circ}$ is a \textit{parameter test element} if $cF^e(z)=0$ for all $e \in \N$ whenever $z \in 0^{*}_{H^d_{\m}(R)}$. 

\item 
The \textit{parameter test submodule} $\tau(\omega_R)$ is defined by $\tau(\omega_R)=\mathrm{Ann}_{\omega_R} 0^{*}_{H^d_{\m}(R)} \subseteq \omega_R$. 
\end{enumerate}
\end{defn}

\begin{rem}
(1) By the definition of $F$-rational rings,  $R$ is $F$-rational if and only if $R$ is Cohen-Macaulay and  $0^{*}_{H^d_{\m}(R)}=0$ (equivalently, $\tau(\omega_R)=\omega_R$). 

(2) (\cite{Sm}) When $R$ is analytically irreducible, $0^{*}_{H^d_{\m}(R)}$ is the unique maximal proper $R$-submodule of $H^d_{\m}(R)$ stable with the Frobenius action $F$. 
\end{rem}

Let $R$ be a ring of prime characteristic $p$. 
Given an $R$-module $M$, we will use $F^e_*M$ to denote the $R$-module which agrees with $M$ as an additive group, but where the multiplication is defined by $r \cdot m =r^{p^e} m$. 
This notation is justified since if $X=\Spec R$ and $\mathcal{M}$ is the quasi-coherent $\sO_X$-module corresponding to $M$, then the push-forward $F^e_*\mathcal{M}$ is the quasi-coherent
sheaf corresponding to $F^e_*M$. 
We say that $R$ is \textit{$F$-finite} if $F^1_*R$ is a finitely generated $R$-module. 

Let $F^e:R \to F^e_*R$ be the $e$-th iteration of the Frobenius map. 
The $e$-th iteration $\mathrm{Tr}^e_R:F^e_*\omega_R \to \omega_R$ of the \textit{trace map} is the $\omega_R$-dual of $F^e$. 

\begin{lem}\label{parameter test submodule}
Let $(R,\m)$ be an $F$-finite reduced local ring of prime characteristic $p$. 
\begin{enumerate}
\item $($\cite[Proposition 3.2 (e)]{Bl2}$)$ $\tau(\omega_R)_P=\tau(\omega_{R_{P}})$ for all $P \in \Spec R$. 
\item Let $(R, \m) \hookrightarrow (S, \n)$ be a flat local homomorphism of $F$-finite reduced local rings of characteristic $p>0$. 
If $S/\m S$ is a field which is a separable algebraic extension of $R/\m$, then $\tau(\omega_R) \otimes_R S=\tau(\omega_S)$. 
\end{enumerate}
\end{lem}

\begin{proof}
We include the proof of (2), because a reference to this fact is not easily found. 
Since $\tau(\omega_S) \otimes_S \widehat{S} \cong \tau(\omega_{\widehat{S}})$ where $\widehat{S}$ denotes the $\mathfrak{n}$-adic completion of $S$ (cf.~ \cite[Proposition 3.2]{HT}), we may assume that $S$ is complete.
First we will show the following claim. 

\begin{cl}
The natural map $F^e_*R \otimes_R S \to F^e_*S$ is an isomorphism for each $e \ge 1$.  
\end{cl}
\begin{proof}[Proof of Claim]
Let $k=R/\m$, $K=S/\m S$ and $\widehat{R}$ be the $\m$-adic completion of $R$. 
By Cohen's structure theorem, there exists an embedding $k \to \widehat{R}$. 
Let $T=\widehat{R} \otimes_k K$ and $\widehat{T}$ be the $\m T$-adic completion of $T$. 
Then by \cite[Proposition 3.4]{Scho}, there exists an isomorphism of $R$-algebras $S \cong \widehat{T}$. 
It immediately follows from the $F$-finiteness of $R$ (resp. $T$) that $F^e_* R \otimes_R \widehat{R} \cong F^e_*\widehat{R}$ (resp. $F^e_* T \otimes_T \widehat{T} \cong F^e_*\widehat{T}$). 
Therefore, it suffices to show that $F^e_*\widehat{R} \otimes_{\widehat{R}} T \cong F^e_*T$. 
Since $K/k$ is a separable algebraic extension, we can construct a direct system $\{T_{\lambda}\}_{\lambda \in \Lambda}$ of finite \'{e}tale extensions $T_{\lambda}$ of $\widehat{R}$ such that $\varinjlim_{\lambda \in \Lambda} T_{\lambda}=T$. 
Using the fact that $T_{\lambda}$ is an \'{e}tale extension of $\widehat{R}$, we see that the natural map $F^e_*\widehat{R} \otimes_{\widehat{R}} T_{\lambda} \to F^e_*T_{\lambda}$ is an isomorphism. 
Thus, 
\[
F^e_*\widehat{R} \otimes_{\widehat{R}} T = F^e_*\widehat{R} \otimes_{\widehat{R}} \varinjlim_{\lambda \in \Lambda} T_{\lambda} \cong \varinjlim_{\lambda \in \Lambda} (F^e_*\widehat{R} \otimes_{\widehat{R}} T_{\lambda}) \cong \varinjlim_{\lambda \in \Lambda} F^e_*T_{\lambda} \cong F^e_* T.
\]
\end{proof}

We are now ready to prove the assertion. Let $c \in R^{\circ}$ be a common parameter test element for $R$ and $S$. 
By an argument very similar to the proof of \cite[Lemma 2.1]{HT}, $\tau(\omega_R)$ coincides with the image of 
\[
\bigoplus_{e \ge 0} \mathrm{Tr}_R^e(F^e_*(c \omega_R)) \to \omega_R, 
\]
where $\mathrm{Tr}_R^e: F^e_*\omega_R \to \omega_R$ is the $e$-th iteration of the trace map on $R$. 
Note that $\omega_R \otimes_R S \cong \omega_S$, because $R \hookrightarrow S$ is a flat local homomorphism and its fiber ring is a field. 
By the above claim, one has 
\[
F^e_*\omega_R \otimes_R S \cong F^e_*\omega_R \otimes_{F^e_*R} F^e_*R \otimes_R S \cong 
F^e_*\omega_R \otimes_{F^e_*R} F^e_*S \cong F^e_*\omega_S,
\]
so that $\mathrm{Tr}_R^e \otimes_R S:F^e_*\omega_R \otimes_R S \to \omega_R \otimes_R S$ is isomorphic to the $e$-th iteration $\mathrm{Tr}_S^e:F^e_*\omega_S \to \omega_S$ of the trace map on $S$. 
Thus, under the isomorphism $\omega_R \otimes S \cong \omega_S$, $\tau(\omega_R) \otimes_R S$ is identified with the image of 
\[
\bigoplus_{e \ge 0} \mathrm{Tr}_S^e(F^e_*(c \omega_S)) \to \omega_S, 
\]
which coincides with $\tau(\omega_S)$ by an argument similar to the proof of \cite[Lemma 2.1]{HT} again. 
\end{proof}

We define the notion of $F$-singularities in characteristic zero, using reduction from characteristic zero to positive characteristic. 

Let $(x \in X)$ be a singularity over a field $k$ of characteristic zero. 
Choosing a suitable finitely generated $\Z$-subalgebra $A \subseteq k$, we can construct a (non-closed) point $x_A$ of a scheme $X_A$ of finite type over $A$ such that $(X_A, x_A) \times_{\Spec A} \Spec k \cong (X, x)$.  
By the generic freeness, we may assume that $X_A$ and $x_A$ are flat over $\Spec A$. 
We refer to $(X_A, x_A)$ as a \textit{model} of $(X, x)$ over $A$. 
Given a closed point $\mu \in \Spec A$, we denote by $(X_{\mu}, x_{\mu})$ the fiber of $(X, x)$ over $\mu$. 
Then $X_{\mu}$ is a scheme of finite type over the residue field $\kappa(\mu)$ of $\mu$, which is a finite field. 

Similarly, we can reduce a projective resolution $\pi:Y \to X$ of $(x \in X)$ and divisors on $Y$ to positive characteristic. 
The reader is referred to \cite[Chapter 2]{HH} and \cite[Section 3.2]{MuS} for more detail on reduction from characteristic zero to characteristic $p$. 

\begin{defn}
Let $(x \in X)$ be a singularity over a field $k$ of characteristic zero. 
Let $\mathbf{P}$ denote a property of local rings essentially of finite type over finite fields which is stable under finite field extensions. 
Suppose that we are given a model of $(X, x)$ over a finitely generated $\Z$-subalgebra $A$ of $k$. 
\renewcommand{\labelenumi}{(\roman{enumi})}
\begin{enumerate}
\item $(x \in X)$ is said to be of $\mathbf{P}$ \textit{type} if there exists a dense open subset $S \subseteq \Spec A$ such that $\sO_{{X_\mu}, {x_{\mu}}}$ satisfies the property $\mathbf{P}$ for all closed points $\mu \in S$. 
\item $(x \in X)$ is said to be of \textit{dense $\mathbf{P}$ type} if there exists a dense subset of closed points $S \subseteq \Spec A$ such that $\sO_{{X_\mu}, {x_{\mu}}}$ satisfies the property $\mathbf{P}$ for all $\mu \in S$. 
\end{enumerate}
\end{defn}

This definition is independent of the choice of representative of the germ and the choice of its model. 

\begin{thm}\label{correspondence}
Let $(x \in X)$ be a reduced singularity over a field of characteristic zero. 
\begin{enumerate}
\item $($\cite{Ha}, \cite{MS}, \cite{Sm}$)$ 
$(x \in X)$ is a rational singularity if and only if it is of $F$-rational type. 
\item $($\cite{Sc}$)$ If $(x \in X)$ is of dense $F$-injective type, then it is a  Du Bois singularity. 
\end{enumerate}
\end{thm}

The converse of Theorem \ref{correspondence} (2) is also expected to hold. 

\begin{conj}\label{conjDB}
Let $(x \in X)$ be a reduced singularity over an algebraically closed field of characteristic zero. 
If $(x \in X)$ is Du Bois, then it is of dense $F$-injective type. 
\end{conj}

Conjecture \ref{conjDB} is open even in dimension two and is equivalent to another more arithmetic conjecture. 

\begin{conj}[\textup{Weak Ordinarity Conjecture \cite{MuS}}]\label{conjWO}
Let $V$ be an $n$-dimensional smooth projective variety over an algebraically closed field $k$ of characteristic zero. 
Given a model of $V$ over a finitely generated $\Z$-subalgebra $A$ of $k$, there exists a dense subset of closed points $S \subseteq \Spec A$ such that the natural Frobenius action on $H^n(V_{\mu}, \sO_{V_{\mu}})$ is bijective for all $\mu \in S$. 
\end{conj}

\begin{prop}[\textup{\cite[Theorem 4.2]{BST}}]
Conjecture \ref{conjDB} holds if and only if Conjecture \ref{conjWO} holds. 
\end{prop}

\subsection{$p$-linear maps on vector spaces}
In this subsection, we recall some basic facts on $p$-linear maps on vector spaces. 
The reader is referred to \cite[Lemma 3.3]{CL} for the proofs. 

Let $k$ be a perfect field of characteristic $p>0$ and $V$ be a finite-dimensional vector space over $k$. 
Let $\varphi: V \to V$ be a $p$-linear map, that is, a morphism of abelian groups such that $\varphi(cv)=c^p \varphi(v)$ for all $c \in k$ and $v \in V$. 
Then $V$ can be uniquely decomposed into a direct sum of two subspaces $V=V_{\rm ss} \oplus V_{\rm nil}$, where $\varphi$ is bijective on $V_{\rm ss}$ and is nilpotent on $V_{\rm nil}$. 
The subspace $V_{\rm ss}$ is called the \textit{semi-simple part} and $V_{\rm nil}$ is called the \textit{nilpotent part} of $V$ with respect to $\varphi$. 

Let $\overline{k}$ be an algebraic closure of $k$ and put $\overline{V}:=V \otimes_k \overline{k}$. 
The map $\varphi$ induces a $p$-linear map  $\overline{\varphi}:\overline{V} \to \overline{V}$ defined by $\overline{\varphi}(v \otimes \lambda)=\varphi(v) \otimes \lambda^p$. 
We consider a morphism of abelian groups $\mathrm{id}-\overline{\varphi}: \overline{V} \to \overline{V}$, where $\mathrm{id}$ is the identity morphism on $\overline{V}$. 
This morphism is surjective and the kernel $\Ker(\mathrm{id}-\overline{\varphi})=\{\overline{v} \in \overline{V} \mid \overline{\varphi}(\overline{v})=\overline{v} \}$ is an $\F_p$-vector subspace of $\overline{V}$ such that 
\[
V_{\rm ss} \otimes_k \overline{k} =\Ker(\mathrm{id}-\overline{\varphi}) \otimes_{\F_p} \overline{k}.
\]
Thus, $V_{ss}=0$ if and only if $\mathrm{id}-\overline{\varphi}$ is injective (equivalently, bijective). 


\section{$F$-nilpotence}
In this section, we study ring-theoretic properties of $F$-nilpotent rings, a class of $F$-singularities defined by the nilpotence of the Frobenius actions on the local cohomology modules.\footnote{
Blickle and Bondu introduced in \cite{BB} the notion of rings close to $F$-rational in order to study Lyubeznik numbers in terms of \'{e}tale cohomology.  
In this paper, we use the term ``$F$-nilpotent rings" for the same notion to emphasize the nilpotence of the Frobenius actions on the local cohomology modules.}
\begin{defn}[\textup{cf.~\cite[Definition 4.1]{BB}}]\label{F-nil def}
Let $(R, \m)$ be a $d$-dimensional local ring of characteristic $p>0$. 
We say that $R$ is \textit{$F$-nilpotent} if the natural Frobenius actions $F$ on $H^0_{\m}(R), \dots, H^{d-1}_{\m}(R)$, $0^*_{H^d_{\m}(R)}$ are all nilpotent. 
A singularity $(x \in X)$ over a field of characteristic $p>0$ is said to be $F$-nilpotent if $\sO_{X, x}$ is $F$-nilpotent. 
\end{defn}

\begin{rem}
By \cite[Proposition 4.4]{Lyu}, the natural Frobenius action $F$ on $H^i_{\m}(R)$ (resp. $0^*_{H^d_{\m}(R)}$) is nilpotent, that is, $F^e(H^i_{\m}(R))=0$ (resp. $F^e(0^*_{H^d_{\m}(R)})=0$) for some $e \in \N$ if and only if for every element $z \in H^i_{\m}(R)$ (resp. $z \in 0^*_{H^d_{\m}(R)}$), there exists $e_z \in \N$ such that $F^{e_z}(z)=0$. 
\end{rem}

The following is the dual statement to Definition \ref{F-nil def}. 
\begin{lem}\label{characterization}
Let $(R, \m)$ be a $d$-dimensional $F$-finite local ring of characteristic $p>0$ with a normalized dualizing complex $\omega_R^{\sbt}$. 
Then $R$ is $F$-nilpotent if and only if there exists $e \in \N$ such that $h^{-i}F^e_*\omega_R^{\sbt} \to h^{-i} \omega_R^{\sbt}$ is a zero map for all $i \ne d$ and $F^e_*\left(\omega_R/\tau(\omega_R) \right) \to \omega_R/\tau(\omega_R)$ is also a zero map. 
\end{lem}
\begin{proof}
It follows from an argument similar to the proof of \cite[Proposition 4.3]{Sc} that for each $i$, $F^e:H^i_{\m}(R) \to H^i_{\m}(R)$ is a zero map if and only if $h^{-i}F^e_*\omega_R^{\sbt} \to h^{-i} \omega_R^{\sbt}$ is a zero map. 
On the other hand, since $\omega_R/\tau(\omega_R)$ is the Matlis dual of $0^*_{H^d_{\m}(R)}$, it follows from an argument similar to the proof of \cite[Lemma 2.1]{HT} that $F^e: 0^*_{H^d_{\m}(R)} \to 0^*_{H^d_{\m}(R)}$ is a zero map if and only if $F^e_*\left(\omega_R/\tau(\omega_R) \right) \to \omega_R/\tau(\omega_R)$ is a zero map. 
Thus, we obtain the assertion. 
\end{proof}

We collect basic properties of $F$-nilpotent rings. 
\begin{prop}\label{basic}
Let $(R, \m)$ be a $d$-dimensional reduced local ring of characteristic $p>0$. 
\begin{enumerate}
\item $R$ is $F$-rational if and only if $R$ is $F$-nilpotent and $F$-injective. 
\item Suppose that $R=S/I$, where $S$ is an $n$-dimensional $F$-finite regular local ring and $I$ is a prime ideal of $S$. 
Let $D_S$ be the ring of differential operators of $S$. 
Then $R$ is $F$-nilpotent if and only if $H^{n-d}_I(S)$ is a simple $D_S$-module and $H^i_I(S)=0$ for all $i \ne n-d$. 
\item If $R$ is $F$-finite and $F$-nilpotent, then so is $R_P$ for all $P \in \Spec R$. 
\item 
Let $(R, \m) \hookrightarrow (S, \n)$ be a flat local homomorphism of $F$-finite reduced local rings of characteristic $p>0$. 
Suppose that $S/\m S$ is a field which is a separable algebraic extension of $R/\m$. 
Then $R$ is $F$-nilpotent if and only if so is $S$. 
\end{enumerate}
\end{prop}
\begin{proof}
(1) 
Assume that $R$ is $F$-rational of dimension $d$. Then $H^0_{\m}(R), \dots, H^{d-1}_{\m}(R)$, $0^*_{H^d_{\m}(R)}$ are all zero. Hence, $R$ is obviously $F$-nilpotent. 

We will show the converse implication. Assume that $R$ is $F$-nilpotent and $F$-injective. 
Then the natural Frobenius actions $F$ on $H^0_{\m}(R), \dots, H^{d-1}_{\m}(R)$, $0^*_{H^d_{\m}(R)}$ are all injective and nilpotent. This means that $H^0_{\m}(R), \dots, H^{d-1}_{\m}(R)$, $0^*_{H^d_{\m}(R)}$ all have to be zero, that is, $R$ is $F$-rational. 

(2)
This is the statement of \cite[Proposition 4.2]{BB}. 

(3) We use the characterization of $F$-nilpotent rings given in Lemma \ref{characterization}. 
The assertion then follows from the fact that the localization of a dualizing complex (resp. a canonical module) of $R$ at $P$ is a dualizing complex (resp. a canonical module) of $R_P$, together with Lemma \ref{parameter  test submodule} (1). 

(4) 
Since $R \hookrightarrow S$ is a flat local homomorphism and its fiber ring is a field, $\omega_S^{\sbt}:=\omega_R^{\sbt} \otimes_R S$ is a normalized dualizing complex of $S$. 
By an argument analogous to the claim in Lemma \ref{parameter test submodule} (2), one has $F^e_*\omega_R^{\sbt}  \otimes_R S \cong F^e_*\omega_S^{\sbt}$. 
On the other hand, it follows from Lemma \ref{parameter test submodule} (2) that $(\omega_R/\tau(\omega_R)) \otimes_R S \cong \omega_S/\tau(\omega_S)$ and $F^e_*(\omega_R/\tau(\omega_R)) \otimes_R S \cong F^e_*(\omega_S/\tau(\omega_S))$. 
The assertion then follows from Lemma \ref{characterization}. 
\end{proof}

\begin{prop}\label{fixed point free}
Let $(R, \m)$ be a reduced local ring containing its algebraically closed residue field $k$ of characteristic $p>0$. 
Suppose that $R$ is equi-dimensional of dimension $d \ge 1$ and the punctured spectrum $\Spec R \setminus \{\m\}$ is $F$-rational. 
Then $R$ is $F$-nilpotent if and only if $F(\xi) \ne \xi$ for all nonzero $\xi \in H^i_{\m}(R)$ and all $i$.  
\end{prop}

\begin{proof}
First we will show the ``only if" part. 
Assume to the contrary that there exist an integer $1 \le i \le d$ and a nonzero element  $\xi \in H^i_{\m}(R)$ such that $F(\xi)=\xi$. 
If $i=d$, then take $c \in R^{\circ}$ such that $c\xi=0$. 
Since $cF^e(\xi)=c \xi=0$ for all $e \in \N$, we see that $\xi \in 0^*_{H^d_{\m}(R)}$. 
Then $F$ is not nilpotent on $0^{*}_{H^d_{\m}(R)}$, because $F^e(\xi)=\xi \ne 0$ for all $e \in \N$. 
Similarly, if $i \le d-1$, then $F$ is not nilpotent on $H^i_{\m}(R)$. 
This is a contradiction. 

Next we will prove the ``if" part. 
Since $\Spec R \setminus \{\m\}$ is $F$-rational, $H^i_{\m}(R)$ and $0^{*}_{H^d_{\m}(R)}$ have finite length as $R$-modules for all $1 \le i \le d-1$. 
We will show that the $H^i_{\m}(R)/0^F_{H^i_{\m}(R)}$ and $0^{*}_{H^d_{\m}(R)}/0^F_{H^d_{\m}(R)}$ are all finite-dimensional $k$-vector spaces. 
We fix an arbitrary element  $\xi \in 0^{*}_{H^d_{\m}(R)}$. 
Take a sufficiently large $n$ so that $\m^n 0^{*}_{H^d_{\m}(R)}=0$, and pick an integer $e$ such that $p^e \ge n$. 
Then $F^e(\m \xi)=\m^{[p^e]}F^e(\xi) \subset \m^{n}F^e(\xi)=0$, so $\m \xi \in 0^F_{H^d_{\m}(R)}$. 
Therefore, $\m \left(0^{*}_{H^d_{\m}(R)}/0^F_{H^d_{\m}(R)}\right)=0$. 
Similarly, we can check that $\m \left({H^i_{\m}(R)}/0^{F}_{H^i_{\m}(R)}\right)=0$ for all $1 \le i \le d-1$. 

Let $V_i:={H^i_{\m}(R)}/0^{F}_{H^i_{\m}(R)}$ for each $i=1, \dots, d-1$ and $V_d:=0^{*}_{H^d_{\m}(R)}/0^F_{H^d_{\m}(R)}$. 
The natural Frobenius action $F$ on $H^i_{\m}(R)$ induces a Frobenius action $\overline{F}$ on $V_i$ for all $i=1, \dots, d$. 
We will prove that $\overline{F}$ is nilpotent on each $V_i$. 
Let $\xi_i \in H^i_{\m}(R)$ whose image $\overline{\xi}_i$ in $V_i$ lies in $\ker \,(\mathrm{id}-\overline{F}: V_i \to V_i)$. 
By definition, $\xi_i-F(\xi_i) \in 0^F_{H^i_{\m}(R)}$, which means that  $F^e(\xi_i)-F^{e+1}(\xi_i)=F^e(\xi_i-F(\xi_i))=0$ for some $e \in \N$. 
It then follows by assumption that $F^e(\xi_i)=0$, so that $\overline{\xi}_i=0$. 
Therefore, $\Ker(\mathrm{id}-\overline{F})=0$, which implies that the semi-simple part of $V_i$ with respect to $\overline{F}$ is trivial, that is, $\overline{F}$ is nilpotent on $V_i$. 

By the definition of $V_i$, the nilpotence of $\overline{F}$ on the $V_i$ is equivalent to the nilpotence of $F$ on $H^0_{\m}(R), \dots, H^{d-1}_{\m}(R)$, $0^*_{H^d_{\m}(R)}$.  Thus, $R$ is $F$-nilpotent. 
\end{proof}

\begin{rem}
We say that a local ring $(R,\m)$ of prime characteristic $p$ is \textit{$F$-fixed point free} if  $F(\xi) \ne \xi$ for all nonzero $\xi \in H^i_{\m}(R)$ and all $i$.  
The notion of $F$-fixed point freeness is closely related to the notion of $F$-instability introduced by Fedder and Watanabe \cite{FW}. For example, suppose that $(R, \m)$ is an $F$-injective complete local ring with algebraically closed residue field $R/\m$. 
Comparing \cite[Corollary 2.8]{En} and \cite[Proposition 5.2]{HS}, we see that $R$ is $F$-unstable if and only if $R$ is $F$-fixed point free. 
\end{rem}

\begin{eg}\label{example charp}
Let $k$ be a perfect field of characteristic $p>0$. 
\begin{enumerate}
\item Let $R=\bigoplus_{n \ge 0}R_n$ be a $d$-dimensional Cohen--Macaulay graded ring with $R_0=k$, and let $\m$ be the unique maximal homogeneous ideal of $R$. 
Suppose that $R$ is $F$-rational away from $\m$, that is, $0^{*}_{H^d_{\m}(R)}$ has finite length as an  $R$-module. 
Since the action of Frobenius on $H^d_{\m}(R)$ multiplies degrees by $p$, we can see that $R_{\m}$ is $F$-nilpotent if and only if the Frobenius action on the part $[H^d_{\m}(R)]_0$ of degree zero is nilpotent. 
In particular, if the $a$-invariant $a(R)$ is negative, then $R_{\m}$ is $F$-nilpotent. 

For example, let $R=k[x,y,z]/(x^2+y^3+z^7)$ and $\m=(x, y, z)$. 
Since the degree zero piece of $H^2_{\m}(R)$ is zero, $R_{\m}$ is $F$-nilpotent. 

\item Let $(R, \m)=\left(k[x,y,z]/(x^2+y^3+z^7+xyz)\right)_{(x, y, z)}$. 
Then $R$ is $F$-injective but not $F$-nilpotent. 
On the other hand, $R/zR \cong k[t^2, t^3]_{(t^2, t^3)}$ is $F$-nilpotent. Thus, $F$-nilpotence does not deform.  

\item (\textup{\cite[Example 5.28]{Bl}})
Let $R=\left(k[x,y,z]/(x^4+y^4+z^4)\right)_{(x,y,z)}$. 
Then $R$ is $F$-nilpotent if and only if $p \not\equiv 1 \; \mathrm{mod} \; 4$. 

\item Suppose that $p=2$ and let $R=k[[s^4, s^3t, st^3, t^4]]$. Then $R$ is $F$-nilpotent but not normal or Cohen-Macaulay. 

\end{enumerate}
\end{eg}


\section{Main result} 
In this section, we propose a conjecture closely related to the weak ordinarity conjecture (Conjecture \ref{conjWO}) and investigate a geometric interpretation of  $F$-nilpotent rings assuming this conjecture.  

An equi-dimensional separated reduced scheme $X$ of finite type over an algebraically closed field $k$ is said to be \textit{a simple normal crossing variety} over $k$ of dimension $n$ if each irreducible component of $X$ is smooth and if for every closed point $x \in X$, there exists an isomorphism 
\[
\widehat{\sO}_{X, x} \cong k[[T_0, \dots, T_n]]/(T_0T_1 \cdots T_{r_x})
\]
of $k$-algebras for some $0 \le r_x \le n$, where $\widehat{\sO}_{X, x}$ is the $\m_{X,x}$-adic completion of $\sO_{X, x}$. 

\begin{conjC}
Let $V$ be an $n$-dimensional projective simple normal crossing variety over an algebraically closed field $k$ of characteristic zero. 
Given a model of $V$ over a finitely generated $\Z$-subalgebra $A$ of $k$, there exists a dense subset of closed points $S \subseteq \Spec A$ such that for every $i$, unless $H^i(V, \sO_V)=0$, the natural Frobenius action on $H^i(V_{\mu}, \sO_{V_{\mu}})$ is non-nilpotent for all $\mu \in S$. 
\end{conjC}

\begin{rem}\label{remark conjN}
(1) It follows from \cite[Lemma 4.5]{BST} that Conjecture \ref{conjWO} implies Conjecture $\mathrm{N}_n$. 
On the other hand, Bhargav Bhatt points out in \cite{Ba} that if Conjecture $\mathrm{N}_n$ holds for all $n$, then Conjecture \ref{conjWO} holds as well. 

(2) Conjecture $\mathrm{N}_{n+1}$ implies Conjecture $\mathrm{N}_n$. 
Indeed, let $V$ be an $n$-dimensional projective simple normal crossing variety over an algebraically closed field $k$ of characteristic zero. 
Suppose that we are given a model of $V$ over a finitely generated $\Z$-subalgebra $A$ of $k$. 
Applying Conjecture $\mathrm{N}_{n+1}$ to the product $V \times \pP^1_k$ of $V$ and the projective line $\pP^1_k$, we see that there exists a dense subset of closed points $S \subseteq \Spec A$ such that for each $i \le n$, provided $H^i(V, \sO_V)=H^i(V \times \pP^1_k, \sO_{V \times \pP^1_k}) \ne 0$, the Frobenius action on $H^i(V_{\mu}, \sO_{V_{\mu}}) =H^i((V \times \pP^1_k)_{\mu}, \sO_{(V \times \pP^1_k)_{\mu}})$ is non-nilpotent for all $\mu \in S$. 
This means that Conjecture $\mathrm{N}_n$ holds. 
\end{rem}

The authors learned the following lemma from Bhargav Bhatt, who they thank. 

\begin{lem}\label{p-adicHodge}
Let $X$ be a smooth projective variety over a number field $K$. 
\begin{enumerate}
\item
Assume that $H^1(X, \sO_X) \ne 0$. 
Then there exist a finite extension $L/K$ of number fields and a set $S$ of primes of density one in $L$ such that the natural Frobenius action on the sheaf cohomology $H^1\left((X \otimes_K L)_{\nu}, \sO_{(X \otimes_K L)_{\nu}}\right)$ of the reduction $(X \otimes_K L)_{\nu}$ of $X \otimes_K L$ is non-nilpotent for all $\nu \in S$. 

\item 
Assume that $H^2(X, \sO_X) \ne 0$.  
Then there exists a set $S$ of primes of positive density in $K$ such that the natural Frobenius action on the sheaf cohomology $H^2(X_{\nu}, \sO_{X_{\nu}})$ of the reduction $X_{\nu}$ of $X$ is non-nilpotent for all $\nu \in S$. 
\end{enumerate}
\end{lem}
\begin{proof}
(1) 
Let $A:=\mathrm{Alb}(X)$ be the Albanese variety of $X$. 
It follows from \cite{Se} and \cite[Theorem 6.6.2]{JR} 
that there exist a finite extension $L/K$ of number fields and a set $S$ of primes of density one in $L$ such that the $p$-rank of $(A \otimes_K L)_{\nu}$ is at least one for all $\nu \in S$. 
The $p$-rank of $(A \otimes_K L)_{\nu}$ is equal to the dimension of the semi-simple part of 
\[H^1((A \otimes_K L)_{\nu}, \sO_{(A \otimes_K L)_{\nu}})=H^1((X \otimes_K L)_{\nu}, \sO_{(X \otimes_K L)_{\nu}})\] with respect to the Frobenius action. 
Thus, the Frobenius action on $H^1((X \otimes_K L)_{\nu}, \sO_{(X \otimes_K L)_{\nu}})$ is non-nilpotent for all $\nu \in S$. 

(2) Ogus proved in \cite{Og} that every abelian surface has infinitely many primes of ordinary reduction. 
We employ the same argument, but we use $p$-adic Hodge theory instead of the semi-simplicity of crystalline Frobenius. 

Let $\overline{K}$ be the algebraic closure of $K$ and $\overline{X}=X \otimes_K  \overline{K}$. 
Let $G_K=\mathrm{Gal}(\overline{K}/K)$ be the absolute Galois group of $K$.  
Fix a prime $\ell>d:=\dim_{\Q_{\ell}} H^2(\overline{X}_{\textup{\'et}}, \Q_{\ell})$ (we can take such an $\ell$, because $d$ is independent of the choice of $\ell$) and consider the corresponding Galois representation $\rho_{\ell}:G_K \to \mathrm{Aut}_{\Q_{\ell}}\left(H^2(\overline{X}_{\textup{\'et}}, \Q_{\ell})\right)$. 
The action of $G_K$ leaves the free part $V_{\ell}$ of $H^2(\overline{X}_{\textup{\'et}}, \Z_{\ell})$ fixed, so let $\overline{\rho}_{\ell}$ denote the representation of $G_K$ on $V_{\ell} \otimes \Z/\ell \Z$ induced by $\rho_{\ell}$. 
Take a finite Galois extension $L$ of $\Q$, containing $K$ and all the $\ell$-th roots of unity, such that $\overline{\rho}_{\ell}(\sigma)=1$ for all $\sigma$ in the absolute Galois group $G_L$ of $L$. 
Let $\nu$ be a prime of $L$ lying over a rational prime $p \ne \ell$ such that $p$ completely splits in $L$. 
Assume in addition that $\nu$ is a place of good reduction for $X \otimes_K L$, and write $(X \otimes_K L)_{\nu}$ for the reduction of $X \otimes_K L$ at $\nu$. 
Note that the set of such primes $\nu$ has density one in $L$. 
After replacing $K$ by $L$ and $X$ by $X \otimes_K L$, we will show that the natural Frobenius action on $H^2(X_{\nu},  \sO_{X_{\nu}})$ is not nilpotent for a set of primes $\nu$ of positive density in $K$. 

Let $t_{\nu}$ be the trace of the $\ell$-adic Frobenius $F_{\nu}$ acting on $H^2(\overline{X}_{\textup{\'et}}, \Q_{\ell})$, which is a rational integer. 
Since $\overline{\rho}_{\ell}|_{G_L}$ is trivial, we have 
\begin{subequations}
\begin{equation}\label{1st equation}
t_{\nu} \equiv d \ \mathrm{mod} \ \ell. 
\end{equation} 
Also, since $p$ splits completely in $L$ and $L$ contains all the $\ell$-th roots of unity, 
we have 
\begin{equation}\label{2nd equation}
p \equiv 1 \; \mathrm{mod} \; \ell.
\end{equation}
\end{subequations}

Suppose to the contrary that the Frobenius action on $H^2(X_{\nu}, \sO_{X_{\nu}})$ is nilpotent for a set of primes $\nu$ of density one. 
It then follows from \cite[Proposition 5.1]{HS} that 
\[H^2(\left(X_{\nu} \otimes_{k_{\nu}} \overline{k_{\nu}}\right)_{\textup{\'et}}, \Z/p \Z)=0,\]
where $\overline{k_{\nu}}$ is the algebraic closure of the residue field $k_{\nu}$ at the prime $\nu$.   
The \'etale cohomology $H^2(\left(X_{\nu} \otimes_{k_{\nu}} \overline{k_{\nu}}\right)_{\textup{\'et}}, \Z/p \Z)$ corresponds to the part of the crystalline cohomology $H^2_{\textup{cris}}(X_{\nu}/W(k_{\nu})) \otimes \Q_p$ on which the eigenvalues of the crystalline Frobenius are $p$-adic units. 
Therefore, the trace of the crystalline Frobenius on $H^2_{\textup{cris}}(X_{\nu}/W(k_{\nu})) \otimes \Q_p$, which is a rational integer, is divisible by $p$.
On the other hand, by Katz-Messing theorem \cite{KM}, the characteristic polynomials of the crystalline Frobenius on $H^2_{\textup{cris}}(X_{\nu}/W(k_{\nu})) \otimes \Q_p$ and of the $\ell$-adic Frobenius $F_{\nu}$ on $H^2(\overline{X}_{\textup{\'et}}, \Q_{\ell})$ are equal to each other. 
Thus, the trace $t_{\nu}$ of the $\ell$-adic Frobenius $F_{\nu}$ is divisible by $p$. 

By the Weil conjectures, the eigenvalues of $F_{\nu}$ are algebraic integers $a_1, \dots, a_d$ such that $|a_i|=p$ for all archimedean absolute values and for all $i=1, \dots, d$. 
In particular, $|t_{\nu}| \le dp$. 
Since we have seen above that $t_{\nu}$ is a rational integer divisible by $p$, it follows from the equations (\ref{1st equation}) and (\ref{2nd equation}), together with the fact that $\ell >d$, that $t_{\nu}=\pm dp$. 
This implies that the $a_i$ are all equal to $\pm p$. 
By assumption, this holds for the set of primes $\nu$ of density one, so Chebotarev's density theorem tells us that $\mathrm{tr}(\rho_{\ell})=\mathrm{tr}\left(\Q_{\ell}(-1)^{\oplus d} \right)$.
Thus, the semi-simplification $\rho_{\ell}^{\rm ss}$ of $\rho_{\ell}$ is isomorphic to $\Q_{\ell}(-1)^{\oplus d}$, because a semi-simple representation is determined by its trace. 

Now we use the Hodge-Tate decomposition of $p$-adic \'etale cohomology due to Faltings \cite{Fa}. 
Fix a maximal ideal $\lambda$ of $\sO_K$ dividing $\ell$, and let $K_{\lambda}$ be the $\lambda$-adic completion of $K$ and $\overline{K_{\lambda}}$ be its algebraic closure.  
Let also $X_{K_{\lambda}}=X \otimes_K K_{\lambda}$.  
We fix an inclusion $\overline{K} \hookrightarrow \overline{K_{\lambda}}$ and consider $\mathrm{Gal}(\overline{K_{\lambda}}/K_{\lambda})$ as a subgroup of $G_K$. 
Then there exists a $\mathrm{Gal}(\overline{K_{\lambda}}/K_{\lambda})$-equivariant isomorphism 
\[
H^2(\overline{X}_{\textup{\'et}}, \Q_{\ell}) \otimes_{\Q_{\ell}} \C_{\ell} \cong \bigoplus_{i+j=2} H^i(X_{K_{\lambda}}, \Omega_{X_{K_{\lambda}}/K_{\lambda}}^j) \otimes_{K_{\lambda}} \C_{\ell}(-j), 
\]
where $\mathrm{Gal}(\overline{K_{\lambda}}/K_{\lambda})$ acts on the left-hand side diagonally and on $H^i(X_{K_{\lambda}}, \Omega_{X_{K_{\lambda}}/K_{\lambda}}^j)$ trivially.  
Since $\left( \C_{\ell} \right)^{\mathrm{Gal}(\overline{K_{\lambda}}/K_{\lambda})}=K_{\lambda}$ and $\left( \C_{\ell}(j) \right)^{\mathrm{Gal}(\overline{K_{\lambda}}/K_{\lambda})}=0$ for all $j \ne 0$, using the fact that 
the Hodge numbers can be recovered from the semi-simplifications of the $p$-adic Galois representations (see for example \cite[Remark 4.2]{It} or \cite[Proposition 5.1]{Wa}), we see that 
\begin{align*}
\dim_K H^2 (X, \sO_X)=\dim_{K_{\lambda}} (\rho_{\ell}^{\rm ss} \otimes_{\Q_{\ell}} \C_{\ell})^{\mathrm{Gal}(\overline{K_{\lambda}}/K_{\lambda})}
&=\dim_{K_{\lambda}} \left(\C_{\ell}(-1)^{\oplus d} \right)^{\mathrm{Gal}(\overline{K_{\lambda}}/K_{\lambda})}\\
&=0.
\end{align*}
However, this contradicts the assumption that $H^2 (X, \sO_X) \ne 0$. 
\end{proof}

Using Lemma \ref{p-adicHodge}, we can prove Conjectures $N_1$ and $N_2$. 
\begin{prop}\label{conjN2}
Conjecture $\mathrm{N}_n$ holds true if $n \le 2$. 
\end{prop}
\begin{proof}
By Remark \ref{remark conjN} (2), it suffices to consider the case where $n=2$. 
By an argument similar to the proof of \cite[Proposition 5.3]{MS}, we may assume
that $k =\overline{\Q}$ without loss of generality. 

Let $V=\sum_i V_i$ be the irreducible decomposition of $V$, and denote $V_{i_0, \dots, i_p}:=V_{i_0} \cap \cdots \cap V_{i_p}$ for $i_0<\cdots<i_p$. 
Then there exists the following spectral sequence (see \cite[4.13]{Fu}):
\[
E_1^{p,q}=\bigoplus_{i_0<\cdots<i_p}H^q(V_{i_0, \dots, i_p}, \sO_{V_{i_0, \dots, i_p}}) \Rightarrow E^{p+q}=H^{p+q}(V, \sO_V).
\]
This spectral sequence degenerates at the $E_2$ term, because after extending the scalars to $\C$,  its terms and differentials are isomorphic to the zeroth graded pieces of the Hodge filtrations on those of the Mayer-Vietoris spectral sequence 
\[
'E_1^{p,q}=\bigoplus_{i_0<\cdots<i_p}H^q(V_{i_0, \dots, i_p}, \C) \Rightarrow H^{p+q}(V, \C), 
\]
which degenerates at the $E_2$ term. 

First, we consider the case where $E^2=H^2(V, \sO_V) \ne 0$.  
Suppose to the contrary that there exists a dense open subset $T \subseteq \Spec A$ such that the Frobenius action on $H^2(V_{\mu}, \sO_{V_{\mu}})$ is nilpotent for all closed points $\mu \in T$. 
Since we have a surjection $E^2 \to  E_2^{0,2}$ and 
\[
E_2^{0,2} = H(E_{1}^{-1, 2}\to E_{1}^{0,2} \to E_1^{1,2})=\bigoplus_i H^2(V_{i}, \sO_{V_{i}}), 
\]
if $H^2(V_{i}, \sO_{V_{i}}) \ne 0$ for some $i$, then the Frobenius action on $H^2(V_{i, \mu}, \sO_{V_{i, \mu}})$ has to be nilpotent for all closed points $\mu \in T$. 
However, we have already seen in Lemma \ref{p-adicHodge} (2) that Conjecture $\mathrm{N}_2$ holds for smooth projective surfaces, so this is a contradiction. Therefore, $H^2(V_{i}, \sO_{V_{i}})=0$ for all $i$, that is, $E_2^{0,2}=0$. 
In this case, the above spectral sequence induces a surjection $E^2 \to E_2^{1,1}$, where 
\[E_2^{1,1} = H(E_1^{0,1} \to E_1^{1,1} \to E_1^{2,1})=\mathrm{coker} \left(\bigoplus_i H^1(V_{i}, \sO_{V_{i}}) \to \bigoplus_{i<j}H^1(V_{i, j}, \sO_{V_{i, j}})\right).\]
If $E_2^{1,1} \ne 0$,then the Frobenius action on $E_{2, \mu}^{1,1}$ has to be nilpotent for all closed points $\mu \in T$. 
On the other hand, by considering Albanese varieties, one can think of $E_2^{1,1}$ as the first cohomology group $H^1(X, \sO_{X})$ of a positive-dimensional abelian variety $X$ over $k$. 
It follows from Lemma \ref{p-adicHodge} (1) that there exists a dense subset of closed points $S \subseteq \Spec A$ such that 
the Frobenius action on $H^1(X_{\mu},\sO_{X_{\mu}}) = E_{2, \mu}^{1,1}$ is non-nilpotent for all $\mu \in S$. This is a contradiction. 
Therefore, $E_{2}^{1,1}=0$ and then $E^2 \cong E_2^{2, 0}$.
Since 
\[
E_2^{2,0} = H(E_1^{1,0} \to E_1^{2, 0} \to E_1^{3, 0})=\mathrm{coker}\left(\bigoplus_{i<j}H^0(V_{i, j}, \sO_{V_{i, j}}) \to \bigoplus_{i<j<l}H^0(V_{i, j, l}, \sO_{V_{i, j, l}})\right), 
\]
the Frobenius action on $H^2(V_{\mu}, \sO_{V_{\mu}})=E^2_{\mu} \cong E^{2,0}_{2, \mu}$ is bijective for all closed points $\mu \in \Spec A$.  
This is a contradiction again. 
Thus, there exists a dense subset of closed points $S_2 \subset \Spec A$ such that the Frobenius action on $H^2(V_{\mu}, \sO_{V_{\mu}})$ is non-nilpotent for all $\mu \in S_2$. 

Next, we consider the case where $E^1=H^1(V, \sO_V) \ne 0$. 
By the above spectral sequence, we have a surjection $E^1 \to E_2^{0,1}$, where 
\[
E_2^{0,1}=H(E_1^{-1, 1} \to E_1^{0,1} \to E_1^{1,1})=\ker \left(\bigoplus_i H^1(V_{i}, \sO_{V_{i}}) \to \bigoplus_{i<j} H^1(V_{i, j}, \sO_{V_{i, j}})\right).
\]
When $E_2^{0,1} \ne 0$, by considering Albanese varieties, one can think of $E_2^{0,1}$ as the first cohomology group $H^1(Y, \sO_{Y})$ of a positive-dimensional abelian variety $Y$ over $k$. 
It follows from Lemma \ref{p-adicHodge} (1) that there exists a dense subset of closed points $S'\subseteq \Spec A$ such that the Frobenius action on $E_{2, \mu}^{0,1}$ is non-nilpotent for all $\mu \in S'$. 
In particular, the Frobenius action on $H^1(V_{\mu}, \sO_{V_{\mu}})=E^1_{\mu}$ is non-nilpotent for all $\mu \in S'$. 
When $E_2^{0,1}=0$, we have an isomorphism $E^1 \cong E_2^{1, 0}$. 
Since 
\begin{align*}
E_2^{1,0}&=H(E_1^{0, 0} \to E_1^{1,0} \to E_1^{2,0})\\
&=H\left(\bigoplus_i H^0(V_i, \sO_{V_i}) \to \bigoplus_{i<j}H^0(V_{i, j}, \sO_{V_{i, j}}) \to \bigoplus_{i<j<l}H^0(V_{i, j, l}, \sO_{V_{i, j, l}})\right), 
\end{align*}
the Frobenius action on $H^1(V_{\mu}, \sO_{V_{\mu}})=E^1_{\mu} \cong E_2^{1,0}$ is bijective for all closed points $\mu \in \Spec A$.  
Thus, in either case, there exists a dense subset of closed points $S_1 \subseteq \Spec A$ such that the Frobenius action on $H^1(V_{\mu},\sO_{V_{\mu}})$ is non-nilpotent for all $\mu \in S_1$. 

Finally, we consider the case where $H^1(V, \sO_V)$ and $H^2(V, \sO_V)$ are both nonzero.  
Since the density of $S_1$ (resp.~$S_2$) depends on Lemma \ref{p-adicHodge} (1) (resp.~ Lemma \ref{p-adicHodge} (2)), 
after enlarging the $\Z$-subalgebra $A$ of $k$ if necessary, we may assume that  $S_1$ has density one and $S_2$ has positive density. 
Then $S_1 \cap S_2$ is a dense subset of closed points in $\Spec A$, and the Frobenius actions on 
$H^1(V_{\mu}, \sO_{V_{\mu}})$ and on $H^2(V_{\mu}, \sO_{V_{\mu}})$ are both non-nilpotent for all $\mu \in S_1 \cap S_2$. 
\end{proof}

The following is a key result of this paper. We obtain a cohomological characterization of rings of $F$-nilpotent type in the case of isolated singularities. 
\begin{thm}\label{vanishing of H^n-1}
Let $(x \in X)$ be an $n$-dimensional normal singularity over an algebraically closed field $k$ of characteristic zero such that $x$ is an isolated non-rational point of $X$. 
Let $\pi:Y \to X$ be a projective log resolution of $(x \in X)$ and $Z$ be the closed fiber $\pi^{-1}(x)$ with reduced scheme structure. 
Suppose that Conjecture $\mathrm{N}_{n-1}$ holds true.  
Then $(x \in X)$ is of $F$-nilpotent type if and only if $H^{i}(Z, \sO_Z)=0$ for all $i \ge 1$. 
\end{thm}
\begin{proof}
Suppose that $(x \in X)$ is of $F$-nilpotent type. 
Then we are given a model of $(\pi, Z)$ over a finitely generated $\Z$-subalgebra $A$ of $k$ such that  $\sO_{X_{\mu}, x_{\mu}}$ is $F$-nilpotent for all closed points $\mu \in \Spec A$. 
For each $\mu \in \Spec A$, we denote by $(X_{\overline{\mu}}, x_{\overline{\mu}})$ the base change of $(X_{\mu}, x_{\mu})$ to an algebraic closure $\overline{k(\mu)}$ of $k(\mu)$. 
Replacing $\sO_{X_{\overline{\mu}}, x_{\overline{\mu}}}$ by its completion, we may assume that $X_{\overline{\mu}}=\Spec R_{\overline{\mu}}$ where $R_{\overline{\mu}}$ is a complete local ring with algebraically closed residue field $\overline{k(\mu)}$. 
It follows from Proposition \ref{basic} (3) that $R_{\overline{\mu}}$ is $F$-nilpotent. 

We consider the Artin-Schreier sequence in the \'etale topology
\[ 
0 \to \Z/p \Z \to \sO_{Y_{\overline{\mu}}} \xrightarrow{{\rm id}-F} \sO_{Y_{\overline{\mu}}} \to 0, 
\]
which induces the long exact sequence 
\[
\cdots \to H^i\left((Y_{\overline{\mu}} \setminus Z_{\overline{\mu}})_{\textup{\'et}}, \Z/p \Z \right) \to H^i\left(Y_{\overline{\mu}} \setminus Z_{\overline{\mu}}, \sO_{Y_{\overline{\mu}}}\right)
\xrightarrow{{\rm id}-F} H^i\left(Y_{\overline{\mu}} \setminus Z_{\overline{\mu}}, \sO_{Y_{\overline{\mu}}}\right) \to \cdots.
 \]
By \cite[Lemma 1.15]{HS}, the map $\mathrm{id}-F$ is surjective on 
\[
H^0(Y_{\overline{\mu}}  \setminus Z_{\overline{\mu}}, \sO_{Y_{\overline{\mu}}}) \cong H^0(X_{\overline{\mu}}  \setminus \{x_{\overline{\mu}}\}, \sO_{X_{\overline{\mu}}})=R_{\overline{\mu}}.
\]
Note that $H^i(Y_{\overline{\mu}} \setminus Z_{\overline{\mu}}, \sO_{Y_{\overline{\mu}}}) \cong H^i(X_{\overline{\mu}} \setminus \{x_{\overline{\mu}}\}, \sO_{X_{\overline{\mu}}}) \cong H^{i+1}_{\{ x_{\overline{\mu}}\}}(\sO_{X_{\overline{\mu}}})$ for all $i \ge 1$. 
Since $x$ is an isolated non-rational point of $X$, after possibly enlarging $A$, we may assume that $H^{i+1}_{\{ x_{\overline{\mu}}\}}(\sO_{X_{\overline{\mu}}})$ is a finitely generated $R_{\overline{\mu}}$-module 
for all $i  \le n-2$ and for all closed points $\mu \in \Spec A$. 
It then follows from \cite[Lemma 1.15]{HS} and Proposition \ref{fixed point free} that for each $\mu \in \Spec A$, the map $\mathrm{id}-F: H^{i+1}_{\{ x_{\overline{\mu}}\}}(\sO_{X_{\overline{\mu}}}) \to H^{i+1}_{\{ x_{\overline{\mu}}\}}(\sO_{X_{\overline{\mu}}})$ is bijective for all $i \le n-2$ and is injective for $i=n-1$.  
By looking at the above long exact sequence, we see that for all $1 \le i \le n-1$ and for all $\mu \in \Spec A$, 
\begin{subequations}
\begin{equation}\label{vanishing}
H^i((Y_{\overline{\mu}} \setminus Z_{\overline{\mu}})_{\textup{\'et}}, \Z/p \Z)=0. 
\end{equation}
\end{subequations}

On the other hand, the above Artin-Schreier sequence also induces the exact sequence
\[
H^{i-1}_{Z_{\overline{\mu}}}(Y_{\overline{\mu}}, \sO_{Y_{\overline{\mu}}}) \to H^{i}_{Z_{\overline{\mu}}}((Y_{\overline{\mu}})_{\textup{\'et}}, \Z /p \Z)  \to H^{i}_{Z_{\overline{\mu}}}(Y_{\overline{\mu}}, \sO_{Y_{\overline{\mu}}}) 
\]
for all $i \ge 0$. 
By the Grauert-Riemenschneider vanishing theorem, after possibly enlarging $A$, we may assume that $H^{i}_{Z_{\overline{\mu}}}(Y_{\overline{\mu}}, \sO_{Y_{\overline{\mu}}})=0$ for all $i \le n-1$ and for all closed points $\mu \in \Spec A$.  
Hence, $H^{i}_{Z_{\overline{\mu}}}((Y_{\overline{\mu}})_{\textup{\'et}}, \Z /p \Z)=0$ for all $i \le n-1$. 
Applying this to the localization exact sequence  
\[
H^{i}_{Z_{\overline{\mu}}}((Y_{\overline{\mu}})_{\textup{\'et}}, \Z /p \Z) \to 
H^{i}((Y_{\overline{\mu}})_{\textup{\'et}}, \Z /p \Z) \to  
H^{i}((Y_{\overline{\mu}} \setminus Z_{\overline{\mu}})_{\textup{\'et}}, \Z /p \Z)
\] 
together with (\ref{vanishing}), we have the fact that 
\[
H^{i}((Z_{\overline{\mu}})_{\textup{\'et}}, \Z /p \Z) \cong H^{i}((Y_{\overline{\mu}})_{\textup{\'et}}, \Z /p \Z)=0
\]
for all $1 \le i \le n-1$ and all $\mu \in \Spec A$, where the first isomorphism follows from the proper base change theorem for  \'{e}tale cohomology. 
By the Artin-Schreier sequence on $Z_{\overline{\mu}}$ (see also \cite[Proposition 5.1]{HS}), 
\[
H^{i}(Z_{\overline{\mu}}, \sO_{Z_{\overline{\mu}}})_{\rm ss}
\cong H^{i}((Z_{\overline{\mu}})_{\textup{\'et}}, \Z /p \Z) \otimes_{\Z /p \Z} \overline{k(\mu)} =0, 
\]
that is, the Frobenius action on $H^{i}(Z_{\overline{\mu}}, \sO_{Z_{\overline{\mu}}})$ is nilpotent for all $1 \le i \le n-1$ and all $\mu \in \Spec A$. 
It follows from an application of Conjecture $\mathrm{N}_{n-1}$ to $Z$ that there exists a dense subset of closed points $S \subseteq \Spec A$ such that $H^{i}(Z_{\overline{\mu}}, \sO_{Z_{\overline{\mu}}})=0$ for all $1 \le i \le n-1$ and all $\mu \in S$. 
Thus, we conclude that $H^{i}(Z, \sO_Z)=0$ for all $1 \le i \le n-1$. 

For the converse, just reverse the above argument. 
The theorem is proved. 
\end{proof}

\begin{rem}\label{fixed char thm}
By the same argument as that of Theorem \ref{vanishing of H^n-1}, we can prove the following, without assuming Conjecture $\mathrm{N}_{n-1}$. 
For (2), note that the Grauert--Riemenschneider vanishing theorem holds for two-dimensional excellent local rings. 

\begin{enumerate}
\item
Let the notation be the same as in Theorem \ref{vanishing of H^n-1}, and suppose that we are given a model of $(x \in X)$ over a finitely generated $\Z$-subalgebra $A$ of $k$. 
Then $(x \in X)$ is of $F$-nilpotent type if and only if there exists a dense open subset $S \subseteq \Spec A$ such that the Frobenius action on $H^i(Z_{\mu}, \sO_{Z_{\mu}})$ is nilpotent for all $i \ge 1$ and for all closed points $\mu \in S$. 

\item
Let $(R, \m)$ be a two-dimensional $F$-finite normal local ring with algebraically closed residue field of characteristic $p>0$. 
Let $\pi:Y \to X$ be a log resolution of $(x \in X):=(\m \in \Spec R)$ and $Z$ be the closed fiber $\pi^{-1}(x)$ with reduced scheme structure. 
Then $R$ is $F$-nilpotent if and only if the natural Frobenius action on $H^1(Z, \sO_Z)$ is nilpotent. 
In particular, if $Z$ is a tree of smooth rational curves (which is equivalent to saying that $H^1(Z, \sO_Z)=0$), then $R$ is $F$-nilpotent. 
\end{enumerate}
\end{rem}

\begin{rem}
In the setting of Theorem \ref{vanishing of H^n-1}, it suffices to have the following conjecture in order to prove the equivalence of the $F$-nilpotence of $(x \in X)$ and the vanishing of $H^i(Z, \sO_Z)$ for all $i \ge 1$: 

\begin{conjC'}
Let $V$ be an $n$-dimensional projective simple normal crossing variety over an algebraically closed field $k$ of characteristic zero. 
Suppose that we are given a model of $V$ over a finitely generated $\Z$-subalgebra $A$ of $k$. 
If there exists a dense open subset $S \subseteq \Spec A$ such that the natural Frobenius action on $H^i(V_{\mu}, \sO_{V_{\mu}})$ is nilpotent for all closed points $\mu \in S$ and all $i \ge 1$,  then $H^i(V, \sO_V)=0$ for all $i \ge 1$. 
\end{conjC'}
Conjecture $\mathrm{N}_n$ clearly implies Conjecture $\mathrm{N}_n'$, and we suspect that  Conjecture $\mathrm{N}_n'$ is weaker than Conjecture $\mathrm{N}_n$. 
We also remark that Conjecture $\mathrm{N'}_n$ implies Conjecture $\mathrm{H}_n$ by the same argument as in Theorem \ref{main result}.
\end{rem}

As an immediate corollary of Theorem \ref{vanishing of H^n-1}, we show a statement analogous to Proposition \ref{basic} (1), assuming Conjecture $\mathrm{N}_{n-1}$. 
\begin{cor}
\begin{enumerate}
\item
If $(x \in X)$ is a rational singularity, then it is of $F$-nilpotent type. 
\item 
Let $(x \in X)$ be an $n$-dimensional normal isolated singularity over an algebraically closed field $k$ of characteristic zero. 
Suppose that Conjecture $\mathrm{N}_{n-1}$ holds true. 
Then $(x \in X)$ is a rational singularity if and only if it is Du Bois and of $F$-nilpotent type. 
\end{enumerate}
\end{cor}
\begin{proof}
(1)
It is immediate from Theorem \ref{correspondence} (1) and Proposition \ref{basic}. 

(2)
The ``only if" part follows from (1) and \cite{St}, so we will prove the ``if" part. 

Assume that $(x \in X)$ is Du Bois and of $F$-nilpotent type. 
Let $\pi:Y \to X$ be a projective log resolution of $(x \in X)$ such that $\pi|_{Y \setminus E}: Y \setminus E \to X \setminus \{x\}$ is an isomorphism, where $E=\sum_i E_i$ is the exceptional locus of $\pi$.
Since $(x \in X)$ is Du Bois, the natural map $(R^i \pi_*\sO_Y)_x \to H^i(E, \sO_E)$ is an isomorphism for all $i \ge 1$. 
On the other hand, by Theorem \ref{vanishing of H^n-1}, $H^i(E, \sO_E)=0$ for all $i \ge 1$. 
Thus, $(R^i \pi_*\sO_Y)_x=0$ for all $i \ge 1$, which means that $(x \in X)$ is a rational singularity. 
\end{proof}

Theorem \ref{vanishing of H^n-1} suggests the following more Hodge-theoretic (and resolution-free) characterization of isolated singularities of $F$-nilpotent type. 
\begin{conjD}
Let $(x \in X)$ be an $n$-dimensional normal isolated singularity over the field $\C$ of complex numbers $($and then $H^*_{\{x\}}(X_{\rm an}, \C)$ has a canonical mixed Hodge structure due to Steenbrink \cite{St}$)$. 
Then $(x \in X)$ is of $F$-nilpotent type if and only if the zeroth graded piece $\mathrm{Gr}^0_F H^{i}_{\{x\}}(X_{\rm an}, \C)$ of the Hodge filtration vanishes for all $i$.  
\end{conjD}

\begin{thm}\label{main result}
Conjecture $\mathrm{N}_{n-1}$ implies Conjecture $\mathrm{H}_{n}$. 
In particular, by Proposition \ref{conjN2}, Conjecture $\mathrm{H}_3$ holds true. 
That is, if $(x \in X)$ is a three-dimensional normal isolated singularity over $\C$, then $(x \in X)$ is of $F$-nilpotent type if and only if $\mathrm{Gr}^0_F H^{i}_{\{x\}}(X_{\rm an}, \C)=0$ for all $i$. 
\end{thm}
\begin{proof}
We may assume that  $X$ is a contractible Stein space and $x$ is the only singular point of $X$. 
Since $(x \in X)$ is normal, one has $H^{0}_{\{x\}}(X, \C)=H^{1}_{\{x\}}(X, \C)=0$. 
Let $\pi: Y \to X$ be a resolution of singularities such that $E:=\pi^{-1}(\{x\})$ is a simple normal crossing divisor and $\pi|_{Y \setminus E} :Y \setminus E \to X \setminus \{x\}$ is biholomorphic. 
We assume in addition that $\pi$ is projective. 
By Theorem \ref{vanishing of H^n-1}, it is enough to show that 
$H^i(E, \sO_E) \cong \mathrm{Gr}^0_F H^{i+1}_{\{x\}}(X, \C)$ for all $i \ge 1$. 

By \cite[(1.10)]{St}, we have the following exact sequence of mixed Hodge structures 
\[
H^{i}(Y, Y \setminus E; \C) \to H^i(E_{\rm an}, \C) \to H^{i+1}_{\{x\}}(X, \C) \to H^{i+1}(Y, Y \setminus E; \C)
\]
for all $i \ge 1$. Therefore, in order to show that 
\[
H^i(E, \sO_E) \cong \mathrm{Gr}_F^0 H^i(E_{\rm an}, \C) \cong \mathrm{Gr}^0_F H^{i+1}_{\{x\}}(X, \C)\] 
for all $i$, it suffices to prove that $\mathrm{Gr}^0_F H^{i}(Y, Y \setminus E; \C)=0$ for all $i$. 

By Artin's algebraization theorem, there exists a complete algebraic variety $V$ which contains $X$ as an open subset and is smooth outside $x$. 
Let $\widetilde{\pi}: W \to V$ be the resolution of singularities obtained by replacing $X$ by $Y$. 
By excision, we have an isomorphism of mixed Hodge structures $H^{i}(Y, Y \setminus E; \C) \cong H^{i}(W_{\rm an}, (W \setminus E)_{\rm an}; \C)$ for all $i$. 
Note that the Hodge filtration on $H^i((W \setminus E)_{\rm an}, \C)$ coincides with the one induced by the Hodge to de Rham spectral sequence
\[
E^{p,q}_1 = H^q(W, \Omega_{W/\C}^p(\log E)) \Rightarrow E^{p+q}=H^{p+q}((W \setminus E)_{\rm an}, \C),
\]
which degenerates at the $E_1$ term, because $E$ is a simple normal crossing divisor on the complete variety $W$. 
Hence, we obtain natural isomorphisms 
\[
\mathrm{Gr}^0_F H^i(W_{\rm an}, \C)\cong H^i(W, \sO_W) \cong \mathrm{Gr}^0_F H^i((W \setminus E)_{\rm an}, \C) 
\] 
for all $i$. 
Applying this to the exact sequence on $W$
\[
\cdots \to H^{i}(W_{\rm an}, (W \setminus E)_{\rm an}; \C) \to H^{i}(W_{\rm an}, \C) \to H^{i}((W \setminus E)_{\rm an}, \C) \to \cdots,
\]
we conclude that 
\[
\mathrm{Gr}^0_F H^{i}(Y, Y \setminus E; \C) \cong \mathrm{Gr}^0_F H^{i}(W_{\rm an}, (W \setminus E)_{\rm an}; \C)=0.
\]
\end{proof}


\section{A characterization in terms of divisor class groups and Brauer groups}
Using a correspondence of rational singularities and $F$-rational rings (see Theorem \ref{correspondence} (1)), we can reformulate a result of Lipman as follows.

\begin{thm}[\textup{cf.~\cite[Theorem 17.4]{Li}}]\label{Lipman}
Let $(R, \m)$ be a two-dimensional normal local ring essentially of finite type over an algebraically closed field $k$ of characteristic zero. 
Let $\widehat{R}$ denote the $\m$-adic completion of $R$. 
Then $R$ is of $F$-rational type if and only if the divisor class group $\mathrm{Cl}(\widehat{R})$ of $\widehat{R}$ is finite $($or, equivalently, $\mathrm{Cl}(\widehat{R})$ is torsion$)$. 
\end{thm}

As a corollary of Theorem \ref{vanishing of H^n-1}, we give a similar characterization of two-dimensional local rings of $F$-nilpotent type in terms of divisor class groups. 

\begin{thm}\label{divisor class group}
Let $(R, \m)$ be a two-dimensional normal local ring essentially of finite type over an algebraically closed field $k$ of characteristic zero. Let $\widehat{R}$ denote the $\m$-adic completion of $R$. 
Then $R$ is of $F$-nilpotent type if and only if the divisor class group $\mathrm{Cl}(\widehat{R})$ of $\widehat{R}$ does not contain the torsion group $\Q/\Z$. 
\end{thm}

\begin{proof}
Let $\pi:\widetilde{X} \to X:=\Spec R$ be a log resolution of $(R, \m)$ such that $\pi|_{\widetilde{X} \setminus E}:\widetilde{X} \setminus E \to X \setminus \{\m\}$ is an isomorphism, where $E=\sum_i E_i$ denotes the exceptional locus of $\pi$. 
By Theorem \ref{vanishing of H^n-1}, it is enough to show that $H^1(E, \sO_E) \ne 0$ if and only if 
$\mathrm{Cl}(\widehat{R})$ contains $\Q/\Z$.  
Since $H^1(E, \sO_E)$ does not change after the completion of $R$, we may assume that $(R,\m)$ is a complete local ring with algebraically closed residue field. 

First note that $H^1(E, \sO_E) \ne 0$ if and only if the Picard scheme $\mathrm{Pic}^0(E)$ has positive dimension.
Since $E=\sum_i E_i$ is a simple normal crossing divisor, by \cite[p.488]{Ar}, there exists an exact sequence  
\[
0 \to \mathbb{G}_{\textup{m}}^{\alpha} \to \mathrm{Pic}^0(E) \to \prod_i \mathrm{Pic}^0(E_i) \to 0
\]
for some integer $\alpha \ge 0$ and each $\mathrm{Pic}^0(E_i)$ is an abelian variety. 
This means that $\mathrm{Pic}^0(E)$ is a semi-abelian variety, in particular, a divisible abelian  group.    
Therefore, $\mathrm{Pic}^0(E)$ has positive dimension if and only if $\mathrm{Pic}^0(E)$ contains $\Q/\Z$. 
By \cite[Lemma 5.4]{CHR}, there exists an effective divisor $D$ supported on $E$ such that $\mathrm{Pic}^0(X) \cong \mathrm{Pic}^0(D)$. 
It follows from an argument of \cite{Ar} that the kernel of the restriction map $\mathrm{Pic}^0(D) \to \mathrm{Pic}^0(E)$ has a composition series with factors isomorphic to $\mathbb{G}_{\textup{a}}$. 
Hence, we have an exact sequence
\[
0 \to \mathbb{G}_{\textup{a}}^{\beta} \to \mathrm{Pic}^0(X) \to \mathrm{Pic}^0(E) \to 0
\]
for some integer $\beta \ge 0$, where the subjectivity of the map $\mathrm{Pic}^0(X) \to \mathrm{Pic}^0(E)$ follows from \cite[Lemma 5.3]{CHR} (see also \cite[Lemma (14.3)]{Li}).
Since $\mathbb{G}_{\textup{a}}$ is torsion-free, $\mathrm{Pic}^0(E)$ contains $\Q/\Z$ if and only if $\mathrm{Pic}^0(X)$ does. 

By \cite[Proposition (14.2)]{Li}, there is another exact sequence
\[
0 \to \mathrm{Pic}^0(X) \to \mathrm{Cl}(R) \to H \to 0, 
\]
where $H$ is a finite group. 
Comparing the largest divisible subgroup of each group in this sequence, we see that $\mathrm{Pic}^0(X)$ contains $\Q/\Z$ if and only if 
$\mathrm{Cl}(R)$ does.  
Thus, summing up the above arguments, we conclude that $H^1(E, \sO_E) \ne 0$ if and only if $\mathrm{Cl}(R)$ contains $\Q/\Z$.  
\end{proof}

\begin{eg}
\begin{enumerate}
\item Let $(R, \m)=\left(\C[x,y,z]/(x^2+y^3+z^7)\right)_{(x, y, z)}$. 
Then $\mathrm{Cl}(\widehat{R})=\C$, which does not contain $\Q/\Z$, 
and $R$ is of $F$-nilpotent type by Example \ref{example charp} (1).   
\item Let $(R, \m)=\left(\C[x,y,z]/(x^2+y^3+z^7+xyz)\right)_{(x, y, z)}$. 
Then $\mathrm{Cl}(\widehat{R})=\C^{\times}$, which contains $\Q/\Z$ via the map $t \mapsto \mathrm{e}^{2 \pi t \sqrt{-1}}$, and $R$ is not of $F$-nilpotent type by Example \ref{example charp} (2). 
\end{enumerate}
\end{eg}

We also have a similar characterization of three-dimensional normal graded isolated singularities of $F$-nilpotent type in terms of Brauer groups. 
The \textit{cohomological Brauer group} $\mathrm{Br}'(X)$ of a scheme $X$ is defined to be the torsion part of $H^2(X_{\textup{\'{e}t}}, \mathbb{G}_{\textup{m}})$. 

\begin{thm}\label{Brauer}
Let $(R, \m)$ be a three-dimensional standard graded normal domain over an algebraically closed field $k$ of characteristic zero. 
Suppose that the localization $R_{\m}$ of $R$ at the homogeneous maximal ideal $\m$ is an isolated singularity, and put $\mathrm{Spec}^{\circ} R=\Spec R \setminus \{\m\}$. 
Then $R_{\m}$ is of $F$-nilpotent type if and only if neither $\mathrm{Br}'(\mathrm{Spec}^{\circ} R)$ nor $\mathrm{Cl}(R)$ contains $\Q/\Z$. 
\end{thm}
\begin{proof}
Let $X=\Proj R$. In view of Theorem \ref{vanishing of H^n-1}, $R_{\m}$ is of $F$-nilpotent type if and only if $H^i(X, \sO_X)=0$ for $i=1, 2$. 

\begin{cl}
$H^i(X, \sO_X)=0$ for $i=1, 2$ if and only if neither $\mathrm{Br}'(X)$ nor $\mathrm{Pic}(X)$ contains $\Q/\Z$. 
\end{cl}
\begin{proof}[Proof of Claim]
Making use of the Kummer sequence and the smooth base change theorem of \'etale cohomology, we can see that $\mathrm{Br}'(X)$ does not change under extensions of algebraically closed fields of characteristic zero. Therefore, we may assume that $k=\C$. 

It follows from \cite[Proposition 1.3]{Sch} that $\mathrm{Br}'(X)$ is isomorphic to the analytic cohomological Brauer group $\mathrm{Br}'(X_{\rm an})$, the torsion part of $H^2(X_{\rm an},  \sO_{X_{\rm an}}^{\times})$. 
The exponential sequence $0 \to \Z \to \sO_{X_{\rm an}} \to \sO_{X_{\rm an}}^{\times} \to 1$ induces the following exact sequence: 
\[
H^2(X_{\rm an}, \Z) \xrightarrow{\alpha} H^2(X, \sO_X) \to H^2(X_{\rm an},  \sO_{X_{\rm an}}^{\times}) \to H^3(X_{\rm an}, \Z).
\]
Note that $H^i(X_{\rm an}, \Z)$ is finitely generated for every $i$, because $X$ is a projective variety.  
If $H^2(X, \sO_X)=0$, then $H^2(X_{\rm an},  \sO_{X_{\rm an}}^{\times})$ is embedded into the finitely generated abelian group $H^3(X_{\rm an}, \Z)$, so $\mathrm{Br}'(X_{\rm an})$ does not contain $\Q/\Z$. 
Conversely, suppose that $H^2(X, \sO_X) \ne 0$. 
Since we know from Hodge theory that the complexification of the map $\alpha$ is surjective, 
$\Q/\Z$ is contained in $\mathrm{Coker} \; \alpha$, which can be considered as a subgroup of $\mathrm{Br}'(X_{\rm an})$. 
Summing up the above, we see that $H^2(X, \sO_X)=0$ if and only if $\mathrm{Br}'(X)$ does not contain $\Q/\Z$. 
By a similar argument, we also see that $H^1(X, \sO_X)=0$ if and only if $\mathrm{Pic}(X)$ does not contain $\Q/\Z$. 
\end{proof}

It is well-known that $\mathrm{Cl}(R) \cong \mathrm{Pic}(X)/\Z [H]$, where $[H]$ is the class of a hyperplane section, so $\mathrm{Pic}(X)$ contains $\Q/\Z$ if and only if $\mathrm{Cl}(R)$ does.   
Thus, by the above claim, it suffices to show that $\mathrm{Br}'(X)$ contains $\Q/\Z$ if and only if $\mathrm{Br}'(\mathrm{Spec}^{\circ} R)$ does. 
For each prime number $p$, we denote the $p$-power torsion subgroup of an abelian group $M$ by $M\{p\}$. 
By \cite[Proposition 14]{Hoo}, there exists an exact sequence 
\[
0 \to \mathrm{Br}'(X)\{p\} \to \mathrm{Br}'(\mathrm{Spec}^{\circ} R)\{p\} \to G_p \to 0,
\]
where $G_p$ is a finite group. 
Therefore, $\mathrm{Br}'(X)\{p\}$ contains $\Q_p/\Z_p$ if and only if $\mathrm{Br}'(\mathrm{Spec}^{\circ} R)\{p\}$ does. 
Since $\Q/\Z \cong \bigoplus_{p} \Q_p/\Z_p$, this proves what we want. 
\end{proof}


\begin{thebibliography}{99}
\bibitem{Ar} 
M.~Artin, Some numerical criteria for contractability of curves on algebraic surfaces, 
Amer. J. Math. \textbf{84} (1962), 485--496.

\bibitem{Ba}
B.~Bhatt, private communication (2015). 

\bibitem{BST}
B.~Bhatt, K.~Schwede and S.~Takagi, The weak ordinarity conjecture and $F$-singularities, 
arXiv:1307.3763, to appear in Advanced Studies in Pure Mathematics.

\bibitem{Bl}
M.~Blickle, The intersection homology $D$-module in positive characteristic, 
University of Michigan, Dissertation, 2001.

\bibitem{Bl2}
M.~Blickle, Test ideals via algebras of $p^{-e}$-linear maps, J. Algebraic Geom. \textbf{22} (2013) 49--83. 

\bibitem{BB}
M.~Blickle and R.~Bondu, 
Local cohomology multiplicities in terms of \'{e}tale cohomology, 
Ann. Inst. Fourier (Grenoble) \textbf{55} (2005), no. 7, 2239--2256. 

\bibitem{CL}
A.~Chambert-Loir, Cohomologie cristalline: un survol, Exposition. Math. \textbf{16} (1998), 333--382.

\bibitem{CHR}
S.~D.~Cutkosky, J.~Herzog and A.~Reguera, 
Poincar\'{e} series of resolutions of surface singularities, 
Trans. Amer. Math. Soc. \textbf{356} (2004), no. 5, 1833--1874.  

\bibitem{DB}
P.~Du Bois, Complexe de de Rham fltr\'{e} d'une vari\'{e}t\'{e} singuli\`{e}re, Bull. Soc. Math. France \textbf{109} (1981), no. 1, 41--81.

\bibitem{En}
F.~Enescu, Local cohomology and $F$-stability, J. Algebra \textbf{322} (2009), 3063--3077. 

\bibitem{Fa}
G.~Faltings, $p$-adic Hodge theory, J. Amer. Math. Soc. \textbf{1} (1988), 255--299. 

\bibitem{FW}
R.~Fedder and K.-i.~Watanabe, 
A characterization of $F$-regularity in terms of $F$-purity, 
\textit{Commutative algebra (Berkeley, CA, 1987)}, pp.~227--245, Math. Sci. Res. Inst. Publ., 15, Springer, New York, 1989. 

\bibitem{Fu}
O.~Fujino, On isolated log canonical singularities with index one, J. Math. Sci. Univ. Tokyo \textbf{18} (2011), no. 3, 299--323.

\bibitem{Ha}
N.~Hara, A characterization of rational singularities in terms of injectivity of Frobenius maps, 
Amer. J. Math. \textbf{120} (1998), no.~5, 981--996. 

\bibitem{HT}
N.~Hara and S.~Takagi, On a generalization of test ideals, Nagoya Math. J. \textbf{175} (2004), 59--74. 

\bibitem{HS}
R.~Hartshorne and R.~Speiser, Local cohomological dimension in characteristic $p$, Ann. of Math. (2) \textbf{105} (1977), no.~1, 45--79.

\bibitem{HH} M.~Hochster and C.~Huneke, Tight closure in equal characteristic zero, Preprint (1999).

\bibitem{Hoo} R.~T.~Hoobler, 
Generalized class field theory and cyclic algebras, 
\textit{$K$-theory and algebraic geometry: connections with quadratic forms and division algebras (Santa Barbara, CA, 1992)}, pp.~239--263, Proc. Sympos. Pure Math., 58, Part 2, Amer. Math. Soc., Providence, RI, 1995.

\bibitem{It} T.~Ito, 
Stringy Hodge numbers and $p$-adic Hodge theory, 
Compos. Math. \textbf{140} (2004), no. 6, 1499--1517. 

\bibitem{JR}
K.~Joshi and C.~S.~Rajan, Frobenius splitting and ordinarity, arXiv:math/0110070.

\bibitem{KM}
N.~M.~Katz and W.~Messing, 
Some consequences of the Riemann hypothesis for varieties over finite fields, 
Invent. Math. \textbf{23} (1974), 73--77. 

\bibitem{Li} 
J.~Lipman, Rational singularities, with applications to algebraic surfaces and unique factorization, 
Inst. Hautes \'Etudes Sci. Publ. Math., No. 36, 1969, 195--279. 

\bibitem{Lyu}
G.~Lyubeznik, $\mathcal{F}$-modules: an application to local cohomology and $D$-modules in characteristic $p > 0$, J. Reine Angew. Math. \textbf{491} (1997), 65--130. 

\bibitem{MS}
V.~B. Mehta and V.~Srinivas, A characterization of rational singularities, 
Asian J. Math. \textbf{1} (1997), no.~2, 249--271.

\bibitem{MuS}
M.~Musta\c{t}\u{a} and V.~Srinivas, Ordinary varieties and the comparison between multiplier ideals and test ideals,  Nagoya Math. J. \textbf{204} (2011), 125--157.

\bibitem{Og}
A.~Ogus, Hodge cycles and crystalline cohomology, \textit{Hodge Cycles, Motives, and Shimura Varieties}, pp.~357--414, Lecture Notes in Mathematics, 900, Springer-Verlag, Berlin-New York, 1982. 

\bibitem{Scho}
H.~Schoutens, Classifying singularities up to analytic extensions of scalars is smooth, Ann. Pure Appl. Logic \textbf{162} (2011), no.~10, 836--852. 

\bibitem{Sch}
S.~Schr\"{o}er, Topological methods for complex-analytic Brauer groups, Topology \textbf{44} (2005), no. 5, 875--894.

\bibitem{Sc2}
K~ Schwede, A simple characterization of Du Bois singularities, Compos. Math. \textbf{143} (2007), 813--828. 

\bibitem{Sc}
K.~Schwede, $F$-injective singularities are Du Bois, Amer. J. Math. \textbf{131} (2009), 445--473. 

\bibitem{Se}
J.-P.~Serre, Groupes de Lie $\ell$-adiques attach\'{e}s aux courbes elliptiques, 
\textit{1966 Les tendances g\'{e}om\'{e}triques en alg\'{e}bre et th\'{e}orie des nombres}, pp.~239--256,  Centre National de la Recherche Scientifique, Paris. 

\bibitem{Sm}
K.~E.~Smith, $F$-rational rings have rational singularities, Amer. J. Math. \textbf{119} (1997), 159--180.

\bibitem{St}
J.~H.~M.~Steenbrink, Mixed Hodge structures associated with isolated singularities, 
\textit{Singularities, Part 2 (Arcata, Calif., 1981)}, pp.~513--536, Proc. Sympos. Pure Math., 40, Amer. Math. Soc., Providence, RI, 1983.

\bibitem{Wa}
C.-L.~Wang, Cohomology theory in birational geometry,  
J. Differential Geom. \textbf{60} (2002), no. 2, 345--354.

\end{thebibliography}
\end{document}